\documentclass[12pt]{amsart}

\usepackage{amssymb}

\usepackage{anysize}
\marginsize{2cm}{2cm}{2cm}{2cm}

\usepackage[pdftex]{graphicx}

\usepackage{versions}
\usepackage{hyperref}
\hypersetup{
  colorlinks   = true, %Colours links instead of ugly boxes
  urlcolor     = blue, %Colour for external hyperlinks
  linkcolor    = blue, %Colour of internal links
  citecolor   = red %Colour of citations
}
\usepackage{cancel}
\usepackage{yhmath}
\usepackage{enumitem}
\usepackage{stackengine}

\usepackage{lineno}
\newtheorem{theorem}{Theorem}[section]
\newtheorem{proposition}[theorem]{Proposition}
\newtheorem{lemma}[theorem]{Lemma}

\theoremstyle{definition}
\newtheorem{definition}[theorem]{Definition}

\theoremstyle{remark}
\newtheorem{remark}[theorem]{Remark}

\numberwithin{equation}{section}

\newcommand*{\where}{\ \ifnum\currentgrouptype=16 \middle\fi|\ }

\renewcommand{\epsilon}{\varepsilon}
\renewcommand{\phi}{\varphi}
\renewcommand{\kappa}{\varkappa}
\renewcommand{\theta}{\vartheta}

\def\Z{{\mathbb Z}}

\title[Cubulating the sphere]{Cubulating the sphere with many facets}

\author{Sergey Avvakumov}
\author{Alfredo Hubard}

\address{Sergey~Avvakumov, School of Mathematical Sciences, Tel Aviv University, Israel}
\email{savvakumov@gmail.com}

\address{Alfredo~Hubard, Universit\'e Gustave Eiffel, CNRS, LIGM, F-77454 Marne-la-Vall\'ee, France}
\email{alfredo.hubard@univ-eiffel.fr}

\begin{document}

\maketitle

\begin{abstract}
For each $d\geq 3$ we construct cube complexes homeomorphic to the $d$-sphere with $n$ vertices in which the number of facets (assuming $d$ constant) is $\Omega(n^{5/4})$. 

This disproves a conjecture of Kalai's stating that the number of faces (of all dimensions) of cubical spheres is maximized by the boundaries of neighbourly cubical polytopes. The conjecture was already known to be false for $d=3$, $n=64$. Our construction disproves it for all $d\geq 3$ and $n$ sufficiently large. Moreover, since neighborly cubical polytopes have roughly $n (\log n)^{d/2}$ facets, we show that even the order of growth (at least for the number of facets) in the conjecture is wrong.

\end{abstract}

%%%%%%%%%%%%%%%%%%%%%%%%%%%%%%%%%%%%%%%%%%%%%%%%%%%%%%%%%%%%%%%%%%%%%%%%%%%%%%%%
%%%%%%%%%%%%%%%%%%%%%%%%%%%%%%%%%%%%%%%%%%%%%%%%%%%%%%%%%%%%%%%%%%%%%%%%%%%%%%%%
\section{Introduction.}
 Face numbers of PL simpicial triangulations of the $d$-sphere \cite{adiprasito2019}, as well as related combinatorial objects (e.g. \cite{kalai2002}) are remarkably well understood. But outside the realm of simplicial complexes we know very little about face numbers. In this paper we focus on face numbers of cube complexes. By a \emph{cube complex} we mean a union of topological cubes such that the intersection of any two cubes is a face of both, a \emph{cubulation} of a topological space is a cube complex homeomorphic to the space, and a \emph{cubical polytope} is a convex $(d+1)$-polytope such that its $d$-skeleton is a cube complex.

Adin \cite{adin1996new} proved that cubical polytopes satisfy analogues of Dehn--Sommerville relations for simplical polytopes. He also asked whether, as in the the simplicial case, cubical $g$-vectors are always non-negative; Babson, Billera, and Chan  \cite{babson1997neighborly} conjectured further that the closed cone spanned by $g$-vectors of cubical polytopes coincides with the positive orthant. Adin,  Kalmanovich, and Nevo \cite{adin2019cone} proved one direction of this conjecture, that the cone contains the orthant. 

In the 90s, inspired by McMullen's upper bound theorem, Kalai conjectured an upper bound theorem for cube complexes homeomorphic to the sphere (Conjecture 4.2 in \cite{babson1997}), specifically, that among cubulations of the $d$-dimensional sphere $S^d$, with $n=2^k$ vertices, for each $i\leq d$, the ones that maximize the number of $i$-th faces, are those whose $\lfloor \frac{d}{2} \rfloor$-skeleton coincides with that of a $k$-dimensional cube.  Cube complexes satisfying this skeleton condition are called (cube) \emph{neighbourly}. Neighbourly cubulations of the sphere for different values of $n$ were first constructed by  Babson, Billera and Chan \cite{babson1997}. In particular, if $n$ is the number of vertices of such a cubulation of the sphere, then the number of facets is $\Omega(n (\log n)^{d/2})$.
Neighbourly cubic polytopes were constructed by Joswig and Ziegler \cite{joswig1999neighborly}, who also showed that a local modification  (flipping a cube) of one of these polytopes provides a $3$-dimensional counterexample to the aforementioned conjecture of Kalai. Our main result is that the order of growth of facets in sphere cubulations can be much larger than that of neighborly cube complexes.

\begin{theorem}
\label{theorem:main}
For any $d\geq 3$ there exists a constant $c(d)>0$ such that for any $n\geq 2^{d+1}$ there exists a cubulation of the $d$-sphere with at most $n$ vertices and at least $c(d)\cdot n^{5/4}$ facets.
\end{theorem}

For $d=3$, the construction behind Theorem~\ref{theorem:main} is as follows: we start with a cubulation of a closed orientable $2$-surface with $O(n)$ vertices and $\Omega(n^2)$ squares, Lemma~\ref{lemma:n-square-surface}. We then take the cartesian product of the surface with a path of length $k$ and fill the ends of the obtained cylinder with two handlebodies getting a cubulation of $S^3$. The key observation, is that the number of vertices introduced in the handlebodies is polynomial, $O(n^4)$, in $n$, and is independent of $k$. Taking $k=n^3$ we get the required number of vertices and facets. The case $d>3$ of the theorem easily follows from $d=3$.

In the other direction, observe that the number of facets in a $d$ dimensional cube complex of $n$ vertices is at most $\frac{n(n-1)}{2^d}$. Indeed, in a cube complex no two cubes share a diagonal, which is defined by a pair of vertices, and a $d$-cube has $2^{d-1}$-diagonals, which yields the previous bound. If $d=3$, and the cube complex is a pseudo-manifold one can play around with the aforementioned observation, Euler's formula and the pseudo-manifold condition to obtain an upper bound of $\frac{n^2}{24}$ cubes. A subquadratic upper bound remains a challenge, even if we add some natural condition like shellability or convexity. 

The following example is a warm up to our main result and implies that if there is a subquadratic upper bound theorem for cubulations of the sphere, a hypothesis stronger than simply connectedness is required. 

\begin{proposition}
\label{warm-up}
For every dimension $d\geq 2$, there exists a constant $c(d)>0$ such that for any $n\geq 2^d$ there exists a simply connected $d$-dimensional cubical complex with at most $n$ vertices and at least $c(d)\cdot n^2$ facets.
\end{proposition}

\begin{proof}
We first consider the case $d=2$.
Let $K_{m,m}$ be the bipartite graph with $m$ vertices on each side. Consider the cartesian product $K_{m,m}\times K_{m,m}$, it is a cube complex with $4m^2$ vertices and $m^4$ squares.

Pick any vertex $v$ in $K_{m.m}$. There is a basis of $\pi_1(K_{m,m}\times K_{m,m})$ with all of its loops in $\{v\} \times K_{m,m} \cup K_{m,m}\times \{v\}$.
Add a cone $u_1 * (\{v\} \times K_{m,m})$ over $\{v\} \times K_{m,m}$ with an apex $u_1$ to the complex. Likewise, add another cone $u_2 * (K_{m,m}\times \{v\})$ over $K_{m,m}\times \{v\}$ with an apex $u_2$. The obtained space is simply connected, it remains to subdivide the cones to make it into a complex.

Let us describe the subdivision of the cone $u_1 * (\{v\} \times K_{m,m})$, the other cone is subdivided analogously. The graph $\{v\} \times K_{m,m}$ is bipartite. Pick one of its components and split every edge between the vertices of this component and $u_1$ in two. Now every triangle in $u_1 * (\{v\} \times K_{m,m})$ becomes a square, however, some of the squares share two edges. To fix this subdivide every square in five, by adding one square in its center and four more for each side, see Figure~\ref{warm-up}.

The resulting cube complex $K$ has $12m^2 + 2$ vertices and $m^4+10m^2$ squares.

For $d > 2$ take the cartesian product $K\times [0,1]^{d-2}$. It has $2^{d-2}\cdot (12m^2 + 2)$ vertices and $m^4+10m^2$ facets.

\end{proof}

    \begin{figure}[ht]
	\center
	\includegraphics[width=0.33\linewidth]{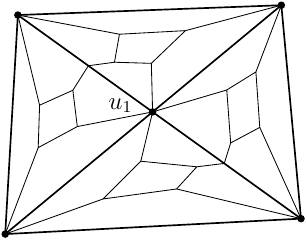}
	\caption{Subdivision of the cone $u_1 * (\{v\} \times K_{m,m})$ for $m=2$.}
	\label{figure:warm-up}
	\end{figure}

\subsection*{Acknowledgments} We thank Roman Karasev and Arnaud de Mesmay for helpful discussions.

%%%%%%%%%%%%%%%%%%%%%%%%%%%%%%%%%%%%%%%%%%%%%%%%%%%%%%%%%%%%%%%%%%%%%%%%%%%%%%%%
%%%%%%%%%%%%%%%%%%%%%%%%%%%%%%%%%%%%%%%%%%%%%%%%%%%%%%%%%%%%%%%%%%%%%%%%%%%%%%%%
\section{Proof of the main result.}

\subsection{Proof of the case $d>3$ of Theorem~\ref{theorem:main} assuming the case $d=3$.}

The proof proceeds by induction with its base being the case $d=3$. Start with a many facet cubulation of $S^d$ given by induction and remove the interior of a facet. We get a cubulation of $B^d$ which we call $Q$. It has $O(n)$ vertices and $\Omega(n^{5/4})$ facets.

The product $Q\times I$ cubulates $B^{d+1}$, where $I$ is the segment cubulated with a single edge. At the same time, $\partial(Q\times I)\times I$ cubulates the cylinder $S^d\times I$.

Glue one copy of $Q\times I$ to each boundary component of $\partial(Q\times I)\times I$ using the identity maps. The resulting space is $S^{d+1}$. It is a cubulation, because each cube of the cylinder $\partial(Q\times I)\times I$ intersects each end of the cylinder at exactly one facet. And each cube of $Q\times I$ intersects each facet of $\partial(Q\times I)$ on a face.

Compared to the cubulation of $S^d$ which we started with, this cubulation of $S^{d+1}$ has exactly $4$ times as many vertices and at least $2$ times as many facets.

It remains to prove the case $d=3$ of Theorem~\ref{theorem:main}.

\subsection{Required lemmas.}

We are going to use the following subdivisions, see Figure~\ref{figure:constant-subdivisions}.
The first is obtained by inserting a square, the second by inserting a cube. The third is a sudivision of a square in an even number (ten) of squares.

	\begin{figure}[ht]
	\center
	\includegraphics[width=0.8\linewidth]{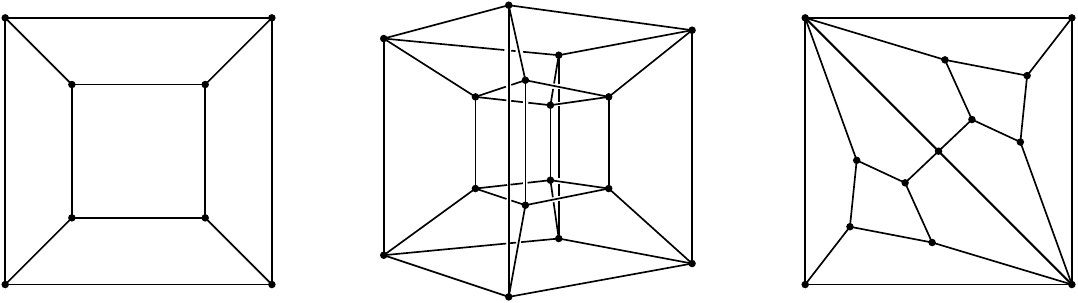}
	\caption{Inserting a square, insetting a cube, a subdivison of a square in ten squares.}
	\label{figure:constant-subdivisions}
	\end{figure}

We postpone the proof of the following lemma until the next section:

\begin{lemma}
\label{lemma:n-square-surface}
For any $n$ there is a cubulation of a closed orientable $2$-surface with $O(n)$ vertices and $\Omega(n^2)$ squares.
\end{lemma}

\begin{remark}Taking the cartesian product of the cubulation of the above lemma with a $(d-2)$-dimensional torus with a standard cubulation we obtain a $d$-manifold with a number of facets which is quadratic in the number of vertices.
\end{remark}

Until the end of the proof of Theorem~\ref{theorem:main}, fix $M$ to be the surface given by Lemma~\ref{lemma:n-square-surface} and $Q$ the corresponding cubulation. We may assume that the number of facets of $Q$ is even, otherwise subdivide one of its squares in ten as shown in Figure~\ref{figure:constant-subdivisions}.

\begin{definition}[Canonical homological basis]
Let $M$ a closed orientable surface of genus $g$. A \emph{canonical homological basis} of $M$ is a set of simple closed curves $\alpha_1,\beta_1,\ldots,\alpha_g,\beta_g$ such that:
\begin{itemize}
\item Curves $\alpha_i, \beta_j$ intersect transversally at a single point if $i=j$, otherwise they are disjoint. Moreover, for $i \neq j$, curves $\alpha_i, \alpha_j$, and curves $\beta_i,\beta_j$ are disjoint.
\item The curves form a basis of the $\Z$-homology of $M$.
\end{itemize}
\end{definition}

\begin{lemma}
\label{lemma:short-basis}
There is a canonical homology basis of $M$ with each curve intersecting every edge of $Q$ at most twice.
\end{lemma}
\begin{proof}
Triangulate $M$ by subdividing each square of $Q$ with a diagonal. The basis constructed right after \cite[Lemma 3]{buser2003triangulations} has the required property -- each of its curves intersects each edge of the triangulation, and hence of $Q$, at most twice.

Let us describe the construction of the basis in a little bit more detail. In the proof of \cite[Lemma 3]{buser2003triangulations} a sequence $X_0,X_1,\ldots,X_g$ is constructed, where each $X_i$ is a triangulated surface of genus $i$ with a hole, and $X_{i+1}$ is obtained from $X_i$ by identifying two pairs $(u_i,u'_i)$ and $(v_i,v'_i)$ of edges in $\partial X_i$. Thus $X_0$ is a fundamental polygon of $M$ and $X_g=M$. Curves $\alpha_{i+1},\beta_{i+1}$ are chosen in $X_i$, very close to $\partial X_i$, and connecting $u_i$ to $u'_i$ and $v_i$ to $v'_i$, respectively. As the curves follow $\partial X_i$ they might intersect every edge in the interior of $X_i$ at most twice -- one time when they pass one endpoint of the edge and another time when they pass its another endpoint.
\end{proof}

\begin{lemma}
\label{lemma:short-skeleton-basis}
There is a subdivision $Q'$ of $Q$ with the following properties:
\begin{itemize}
\item $Q'$ has at most $O(n^4)$ squares.
\item Each edge of $Q$ is subdivided into an even number of edges.
\item $M$ has a canonical homology basis with each curve being an edge path in the $1$-skeleton of $Q'$ and not longer than $O(n^2)$ edges.
\item The neighborhood of every curve of the basis is regular. More precisely, for every curve $\alpha$ of the basis the subcomplex of squares intersecting it is $\alpha\times I_2$, where $I_2$ is a line segment subdivided with two edges. 
\end{itemize}
\end{lemma}

\begin{proof}

First, we subdivide $Q$ satisfying all the requirements except possibly the last one, see Figure~\ref{figure:short-basis}:

	\begin{figure}[ht]
	\center
	\includegraphics[width=1\linewidth]{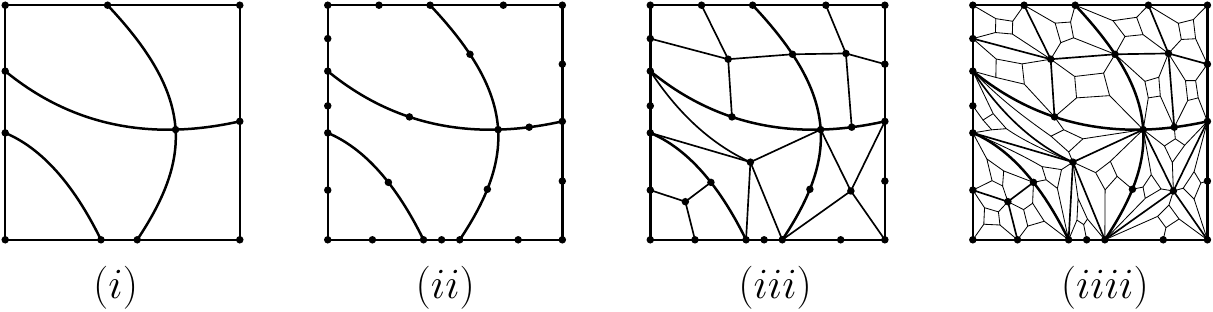}
	\caption{Proof of Lemma~\ref{lemma:short-skeleton-basis}.}
	\label{figure:short-basis}
	\end{figure}

\begin{itemize}
\item[$(i)$] We start with the subdivison given by the edges of $Q$ together with the curves of the basis given by Lemma~\ref{lemma:short-basis}.
\item[$(ii)$] Split every edge of this new subdivision in two. Now every face is a polygon with an even number of edges.
\item[$(iii)$] Split every polygon into squares by adding a new vertex inside of it. 
\item[$(iiii)$] Split every square in $5$ by insetting a square. This guarantees that the result is a cubical complex.
\end{itemize} 

It remains to check that $Q'$ satisfies the required properties. By construction, every edge of $Q$ is subdivided into an even number of edges.

The subdivision at step $(i)$ has at most $O(n^4)$ vertices - each of the $O(n^2)$ basis curves intersects each of the $O(n^2)$ edges of $Q$ at most twice. By the same reasoning it also has at most $O(n^4)$ edges. So, since the genus of $S$ is $O(n^2)$, then by Euler's formula this subdivision has at most $O(n^4)$ faces.

At step $(ii)$ we add at most $O(n^4)$ new edges and vertices. So, this subdivision also has at most $O(n^4)$ vertices, edges, and faces. 

By construction, the subdivision at step $(iii)$ also has at most $O(n^4)$ vertices - the sum of vertices and faces of step $(ii)$. All its faces are squares, so it has twice as many edges as faces. By the Euler formula, we get that it also has $O(n^4)$ edges and faces.

Finally, the subdivision at step $(iiii)$ has $5$ times as many squares, which again makes it $O(n^4)$.
The basis curves stopped changing after step $(ii)$ and each has at most $O(n^2)$ edges.

It remains to satisfy the last requirement. To do that, for every pair of intersecting curves in the basis, we cut the cubulation along this pair and glue back two cubulated ribbons instead, see Figure~\ref{figure:regularization}.
This operation adds $O(n^4)$ squares -- the number of curves, $O(n^2)$, times their length, also $O(n^2)$. Since different pairs of basis curves are disjoint, we can regularize their neighborhoods independently. Each original edge of $Q$ is still subdivided an even number of times.

    \begin{figure}[ht]
	\center
	\includegraphics[width=0.6\linewidth]{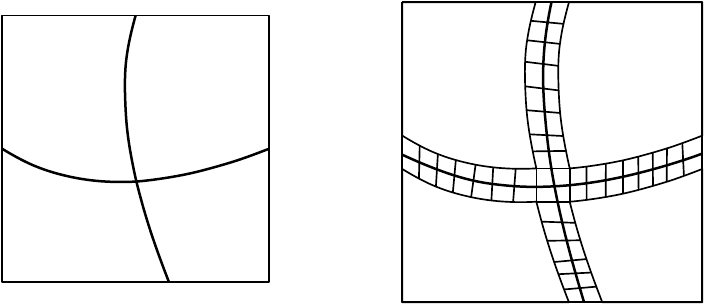}
	\caption{Regularizing neighborhoods of the basis curves.}
	\label{figure:regularization}
	\end{figure}

\end{proof}

The following result can be found in \cite[Theorem 2]{EPPSTEIN19993}.

\begin{lemma}
\label{lemma:fill-sphere}
Any cubulation of $S^2$ with an even number $n$ of squares is the boundary of a cubulation of the $3$-ball with $O(n)$ cubes.
\end{lemma}

\subsection{Proof of Theorem~\ref{theorem:main} for $d=3$.}

Let $I_k$ be the cubulation of the interval $[a,b]$ with $k$ edges.
We begin with the standard embedding $M\times [a,b]\to S^3$ cubulated as $Q\times I_k$ (see Figure~\ref{figure:handlebody}) (an embedding $M\to S^3$ is \emph{standard} if it is a connected sum of standardly embedded and unlinked tori and an embedding $M\times [a,b]\to S^3$ is \emph{standard} if its restriction to $M\times\{a\}$ is standard).

	\begin{figure}[ht]
	\center
	\includegraphics[width=0.80\linewidth]{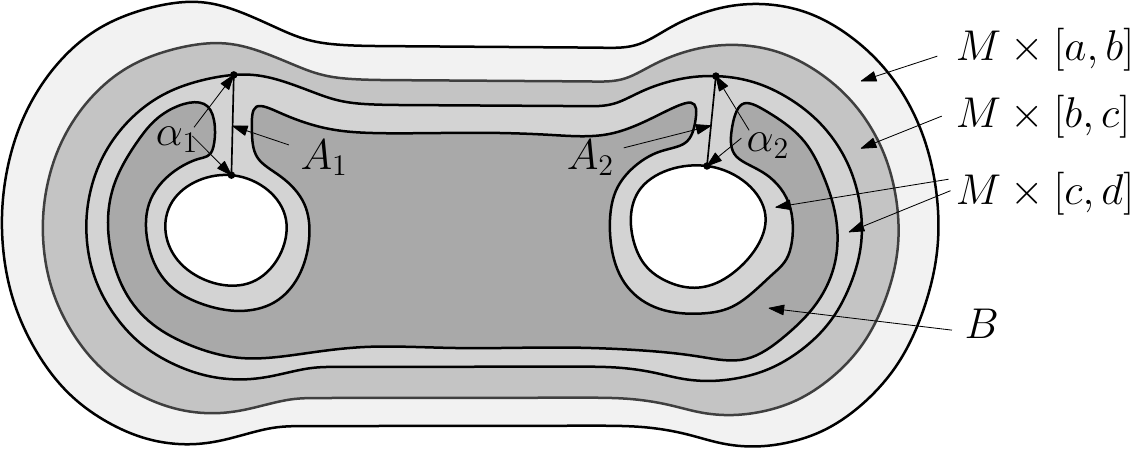}
	\caption{Proof of Theorem~\ref{theorem:main}. }
	\label{figure:handlebody}
	\end{figure}

Now we need to cubulate the two connected components of $S^3 \setminus M\times [a,b]$.

Each of them is handlebody which we will cubulate using the same strategy, so we only describe in detail how we cubulate the component bounded by $M\times \{b\}$.

We start by gluing a \emph{refining cylinder} $M\times [b,c]$ to $M\times \{b\}$. Its ends $M\times \{b\}$ and $M\times \{c\}$ are cubulated as $Q$ and $Q'$, respectively. To be more precise, a small fraction (at most one square in each  square of the original cubulation $Q$) of the squares of $M\times \{c\}$ may be subdivided compared to $Q'$, but this does not affect the properties of $Q'$ given by Lemma~\ref{lemma:short-skeleton-basis} and we keep calling this cubulation $Q'$. The interior of the refining cylinder is cubulated using at most $O(n^4)$ cubes. 

Finally, we cubulate the handlebody bounded by $M\times \{c\}$ using again at most $O(n^4)$ cubes.

We  will describe the cubulations of the refining cylinder and the handle body later.

After cubulating both connected components of $S^3 \setminus M\times [a,b]$, the resulting sphere has at most $k\cdot O(n) + O(n^{4})$ vertices and at least $k\cdot \Omega(n^2)$ cubes. (Here $k\cdot O(n)$ is the number of vertices in $Q\times I_k$, $O(n^{4})$ is the number of cubes, and hence an upper bound on the number of vertices in the handlebodies, and $k\cdot \Omega(n^2)$ is the number of cubes in $Q\times I_k$). Choosing $k=n^{3}$ we get the required number of vertices and cubes.

%It remains to cubulate the refining cylinders and the handlebody. 

{\bf Refining cylinders.} We need to cubulate the cylinder $M\times [b,c]$ so that one of its ends is cubulated as $Q$ and the other as $Q'$.
We start with the cubulation $Q\times [b,c]$ with $[b,c]$ cubulated using a single edge.

As $Q$ is a cubulation, its $1$-skeleton is bipartite. Pick one side of this bipartite graph and for each vertex $v$ in it split the edge $v\times [b,c]$ into two. Now, for every square $F$ in $Q$ the boundary of every side face of $F\times [b,c]$ is a polygon with an even number of edges (recall, $Q'$ subdivides every edge of $Q$ into an even number of edges). So, we can subdivide every side face of $F\times [b,c]$ into squares by adding a vertex in its center. We then subdivide every square into $5$ by insetting a new square, to ensure that we have a cubical complex, see Figure~\ref{figure:refining}. 

	\begin{figure}[ht]
	\center
	\includegraphics[width=0.70\linewidth]{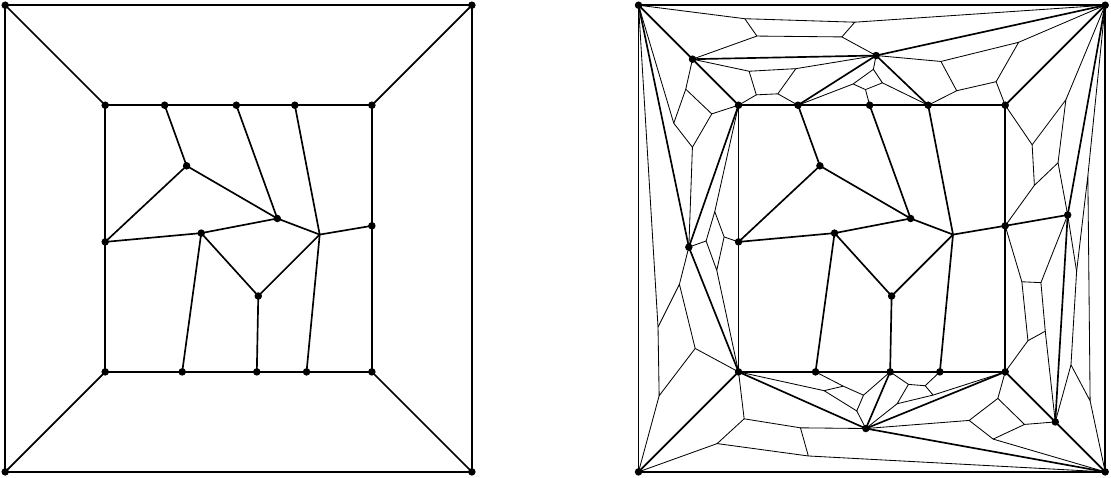}
	\caption{Subdivision of the boundary of $F\times [b,c]$. The outer face on the picture corresponds to $F\times \{b\}$ and is not subdivided, the inner square is subdivided according to $Q'$. Note, that exactly $2$ (top left and bottom right) out of $4$ edges connecting the inner and the outer squares are split in half.}
	\label{figure:refining}
	\end{figure}

Now, consider the boundary complex of every $F\times [b,c]$. If it has an odd number of squares then split one of the squares in $F\times \{c\}$ into an even number of squares as shown in Figure~\ref{figure:constant-subdivisions}. This changes the subdivision $Q'$, but it keeps the properties given by Lemma~\ref{lemma:short-skeleton-basis}, as we do this operation at most once for each $F$.

Now the boundary of each $F\times [b,c]$ is subdivided into an even number of squares.
By Lemma~\ref{lemma:fill-sphere} we can fill this boundary by a cubulated $3$-ball. To ensure that we still get a cubical complex, we cubulate each cube into $7$ smaller cubes by insetting a cube inside of it as shown in Figure~\ref{figure:constant-subdivisions}. The number of cubes we use is at most the number of squares in the boundary of $F\times [b,c]$ times a constant. Note, that by construction the number of squares in the boundary of $F\times [b,c]$ is at most a constant times the number of squares in $F\times \{c\}$, and hence the total number of cubes we need to cubulate the whole refining cylinder is at most the number of squares in $Q'$, which is $O(n^4)$.

Finally, note that since the number of squares in the boundary of each $F\times [b,c]$ is even, then the number of squares in $M\times \{b\}$ and $M\times \{c\}$, i.e., in $Q$ and $Q'$ (which we have changed sightly) have the same parity. Since $Q$ has an even number of squares, so does $Q'$.

{\bf Handlebody.}
Let $\alpha_1,\beta_1,\ldots,\alpha_g,\beta_g$ be the canonical homology basis of $M$ given by Lemma~\ref{lemma:short-skeleton-basis}. We may assume that $\alpha_1,\ldots,\alpha_g$ are the standard meridians and $\beta_1,\ldots,\beta_g$ are the standard parallels of the standardly embedded $M$. Otherwise, we can just use a self-isotopy of $M$ before starting the whole construction.

For each $i$, let $A_i$ be a polygon with an even number $|\alpha_i|\leq O(n^2)$ of edges.
As before, we subdivide it into $|\alpha_i|/2$ squares by adding a vertex in the interior and then subdivide every square in $5$. A cube complex $Q'_{plus\,disks}$ is obtained from $Q'$ by gluing $A_i$ to each $\alpha_i$. This complex cubulates $M_{plus\,disks}:=M\times \{c\}\cup\underset{i}{\bigcup}A_i$.
A cubulation $Q'_{sphere}$ of the $2$-sphere $S^2$ is obtained from $Q'$ by first cutting $Q'$ along each $\alpha_i$ and then gluing a copy of $A_i$ to each of the resulting boundary components.

The cylinder $S^2\times [c,d]$ is cubulated as $Q'_{sphere}\times [c,d]$, where $[c,d]$ is cubulated with a single edge. Now we glue one of its ends, $S^2\times\{c\}$, to $M_{{plus\,disks}}$ by the map identifying the two copies of $A_i$ in $Q'_{sphere}\times\{c\}$ for each $i$. This produces a cube complex, since the basis curves $\alpha_1,\ldots,\alpha_g$ had pariwise disjoint regular neighborhoods in $Q'$. 

The other end, $S^2\times\{d\}$, we fill with a cubulated $3$-ball $B$ using Lemma~\ref{lemma:fill-sphere}, which finishes the cubulating of the handlebody. We can apply the lemma since $Q'$, and hence $Q'_{sphere}\times\{d\}$, have an even number of squares. (This time we do not need to subdivide the cubes of $B$ into $7$ cubes, as each cube of the cylinder $Q'_{sphere}\times [c,d]$ intersects the boundary of the ball $B$ in a single square.)

The number of squares in $Q'_{sphere}$ is at most $O(n^{4})$ - at most $O(n^2)$ squares in each of $O(n^2)$ disks $A_i$ and at most $O(n^{4})$ squares in $Q'$. So, the number of cubes in $Q'_{sphere}\times [c,d]$ and in the cubulation of the ball $B$, and hence in the whole handlebody is at most $O(n^{4})$.

%%%%%%%%%%%%%%%%%%%%%%%%%%%%%%%%%%%%%%%%%%%%%%%%%%%%%%%%%%%%%%%%%%%%%%%%%%%%%%%%
%%%%%%%%%%%%%%%%%%%%%%%%%%%%%%%%%%%%%%%%%%%%%%%%%%%%%%%%%%%%%%%%%%%%%%%%%%%%%%%%
\section{Proof of Lemma~\ref{lemma:n-square-surface}.}

Let $G\to M$ be an embedding of a bipartite graph $G$ into a closed orientable surface such that every connected component of the complement $M\setminus G$ to $G$ is an open disk. The boundary of the closure of every such disk is a cycle in $G$. Note that these cycles may not be simple, however, orientations on the disks give a natural order of the vertices on the cycles.

The following lemma gives a sufficient condition on the cycles which allows us to refine the embedding $G\to M$ to a cubulation without adding too many new vertices:

\begin{lemma}
\label{lemma:cycles}
For any $c>0$ there is $C>0$ such that for any bipartite graph $G$, closed orientable surface $M$, and embedding $G\to M$, the embedding can be refined to a cube complex with at most $|V(G)|\cdot (1+C)$ vertices provided that all the connected components of $M\setminus G$ are open disks and the corresponding cycles in $G$ satisfy the following properties:

\begin{itemize}
    \item [(I)] The set of cycles can be partitioned into ``odd'' and ``even'', so that for any vertices $j,k$ a sequence of three consecutive vertices of the form $j,?,k$ appears at most once in at most one even cycle, and at most once in at most one odd cycle. Here $?$ is a substitute for an arbitrary vertex, i.e., if $j,i_1,k$ appears in an even cycle, then $j,i_2,k$ cannot appear in another place in this cycle or any other even cycle, even if $i_1\neq i_2$.

    \item [(II)] Every path of length at most $10$ in any cycle is simple, i.e., it does not visit any vertex of $G$ more than once.
    
    \item [(III)] The cycles can be split into simple paths, the total number of paths is at most $c|V(G)|$.
\end{itemize}
\end{lemma}

\begin{proof}
First, we describe how we split the open disks in $M\setminus G$ into squares and check that we do not add too many new vertices. Afterwards, we check that the splitting is a cube complex. Let us call the parts of the bipartite graph $G$ \emph{left} and \emph{right}. Denote $n=|V(G)|$.

Consider an odd cycle $S$ in $G$, it is a concatenation $S=P_1+ P_2\ldots+ P_k$ of several simple paths $P_1,\ldots,P_k$ given by property (III). For each $P_i$, let $P'_i\subseteq P_i$ be its longest subpath starting and ending at a right vertex, i.e., at a vertex from the right half of $G$. Clearly, $|P_i|-|P'_i|\leq 2$. Let $Q_i$ be the gap between $P'_i$ and $P'_{i+1}$, i.e., the path in $S$ between the end of $P'_i$ and the start of $P'_{i+1}$. The endpoints of $Q_i$ are right and $|Q_i|\leq 2$, the latter meaning that $Q_i$ is simple by property (II). So, we have represented $S$ as a concatenation $S=P'_1+Q_1+P'_2+Q_2+\ldots + P'_k+Q_k$ of simple paths, each starting and ending at a right vertex. Do this for every odd cycle.

Likewise, we split every even cycle into simple paths, this time starting and ending at a left vertex. By construction, the total number of paths is at most $2cn$, twice the number of paths given by property (III).

Add a vertex in the interior of the connected component of $M\setminus G$ corresponding to each cycle and connect it to the endpoints of the paths in this cycle. This splits the cycle into smaller simple parts which we call \emph{subcycles}, see Figure~\ref{figure:cycle-cub} (left).

For every subcycle, add a vertex in the interior of the disk it bounds. For even subcycles, connect this vertex to all the left vertices in the boundary of the subcycle. For odd subcycles, connect it to all the right vertices in the boundary of the subcycle. This creates one square with vertices at the ends of a path; subdivide this square in $5$, see Figure~\ref{figure:cycle-cub}.

    \begin{figure}[ht]
	\center
	\includegraphics[width=.95\linewidth]{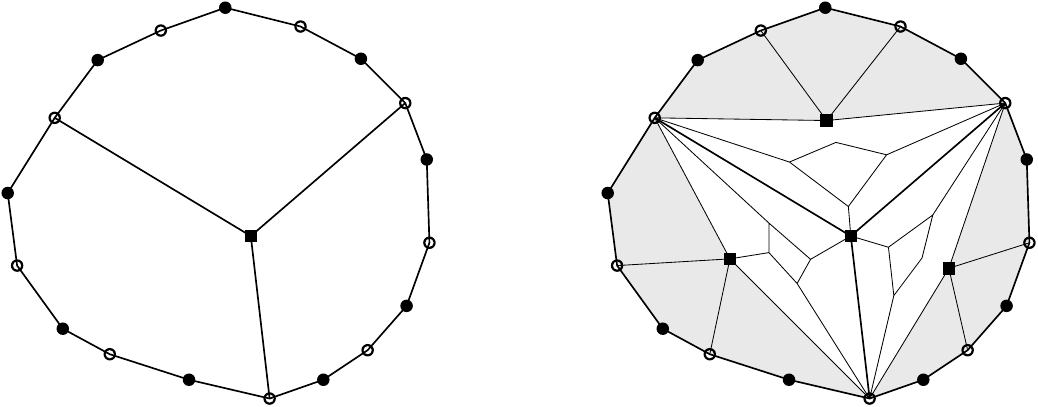}
	\caption{Cubulating an odd cycle consisting of $3$ simple paths starting and ending in right vertices of $G$.  The cycle is first split into $3$ simple subcycles (left), each of them then split into, in this case, $4$ squares (right). The squares which contain both endpoints of one of the three paths are split in $5$. Left vertices of $G$ are black and right are white.}
	\label{figure:cycle-cub}
	\end{figure}

We add $1$ vertex per cycle plus $5$ vertices per subcycle, which amounts to less than $6$ vertices per subscycle. Because the number of subcycles is at most $2cn$, the total number of vertices we add is in $O(n)$.

It remains to check that the intersection of any two squares is a face of both.
The $1$-skeleton of the constructed splitting of $M$ into squares is bipartite.
So, if the intersection of two squares is not a face of one of them, it means that the squares share a diagonal.

It is clear that we only need to consider squares which have three vertices on the boundary of their respective cycle - the gray squares in Figure~\ref{figure:cycle-cub}. It is also clear that since subcycles are simple, two squares resulting from the cubulation of the same subcycle cannot share a diagonal.

So, it remains to consider two squares coming from cubulations of two different subcycles (possibly, but not necessarily, originating from different cycles). Such two squares differ in the vertices lying in the interior of their respective subcycle (black square vertices in Figure~\ref{figure:cycle-cub}). So, they could only share the diagonal with both vertices on the boundary of their respective cycle. 

However, this diagonal uniquely determines the square, so two squares cannot share it. Indeed, this follows by property (I) and because the diagonal determines the parity of the subcycle and hence the cycle to which the square belongs -- even, if its vertices are left and odd, if they are right.

\end{proof}

To prove Lemma~\ref{lemma:n-square-surface}, we construct a bipartite graphs $G$ with $2n$ vertices and $\Omega(n^2)$ edges together with embeddings $G\to M$ satisfying the properties in the hypothesis of Lemma~\ref{lemma:cycles}. By the lemma, it gives us cubulations of $M$ with $O(n)$ vertices. The number of squares will be $\Omega(n^2)$, since it is half the number of edges.

{\bf Construction of the graph $G$.}

Let $G$ be a bipartite graph with an odd prime number $n$ of vertices in each part.
We denote the vertices in each part by the numbers $1,\ldots,n$, so the $i$th vertices in either component is denoted by $i$. This should not cause confusion.

We think of the numbers $1,\ldots,n$ as situated in a circle in \emph{counterclockwise} order. Operations $+$, resp. $-$, mean moving in counterclockwise, resp. clockwise, order. Vertices $i,j$ are said to be in counterclockwise order if $j=i+d$ for some $1\leq d < n/2$.

Take $d_{max}$ to be the largest odd number smaller than $n/10$. For each $i$, vertex $i$ is connected to $i-d_{max}, i-d_{max}+1,\ldots,i-1,i+1,\ldots, i+d_{max}$, i.e., to all the vertices $i\pm d$ with $1\leq d\leq d_{max}$. Hence, $G$ has $\Omega(n^2)$ edges.

{\bf Construction of the surface.}

We take a $2$-disk for each of the vertices of $G$ ($2n$ disks in total), and for each edge of $G$ glue a strip connecting the corresponding pairs of disks.

The strips are glued to the disks in the following way: when we travel counterclockwise around the boundary of the disk corresponding to vertex $i$ we encounter the strips in the following order: $1,-1,2,-2,3,\ldots,d_{max}, -d_{max}$, see Figure~\ref{figure:vertex-star}~(left).

    \begin{figure}[ht]
	\center
	\includegraphics[width=1.0\linewidth]{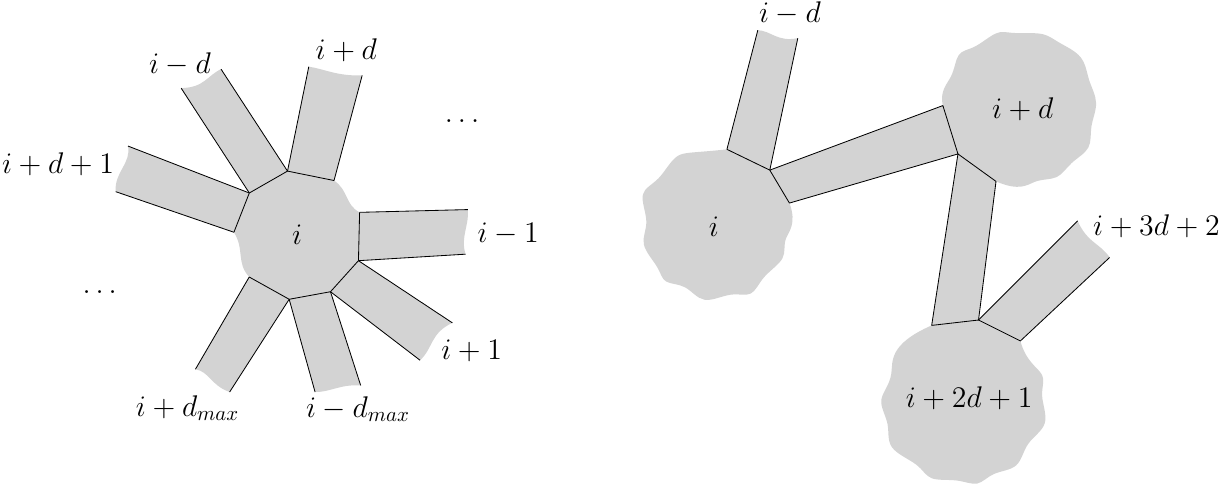}
	\caption{Neighborhood of the disk corresponding to vertex $i$ in $M_0$ (left). Part of $M_0$ (right).}
	\label{figure:vertex-star}
	\end{figure}

What we have now is a $2$-dimensional surface with holes $M_0$, see Figure~\ref{figure:vertex-star}~(right). The surface is orientable since none of the strips are twisted. Glue a disk to each boundary component of the surface to obtain a closed surface $M$. The graph $G$ separates $M$ into polygons glued together along their edges. Each polygon corresponds to a cycle in $G$ and to a boundary component of $M_0$. The cycle might not be simple, which plays some role in the sequel.

{\bf Odd and even cycles.}

Let us describe the boundary components of $M_0$. For a given vertex $i$, the  strip going to any other vertex has \emph{right} and \emph{left} sides with respect to $i$ - going around the disk corresponding to $i$ we encounter the \emph{right} side first. Note, that each side of the strip is right with respect to one of its ends and left with respect to another.

Consider now an arbitrary boundary component of $M_0$. It contains a part of the boundary of the strip going from $i-d$ to $i$ for some $i$ and $d$. There are two possibilities:

Case $1$: the boundary component arrives to $i$ along the right (with respect to $i$) side of the strip, Figure~\ref{figure:even-cycle}.
In this case it then goes to $i+d$ and again arrives along the right side of the strip, and so on. So, it visits the vertices \[i-d, i+d, i+2d, i+3d, \ldots \] We call such boundary component and the corresponding cycle in the graph \emph{even} (because the difference between consecutive vertices of this cycle in the same component of $G$ is even, $2d$).

    \begin{figure}[ht]
	\center
	\includegraphics[width=1.0\linewidth]{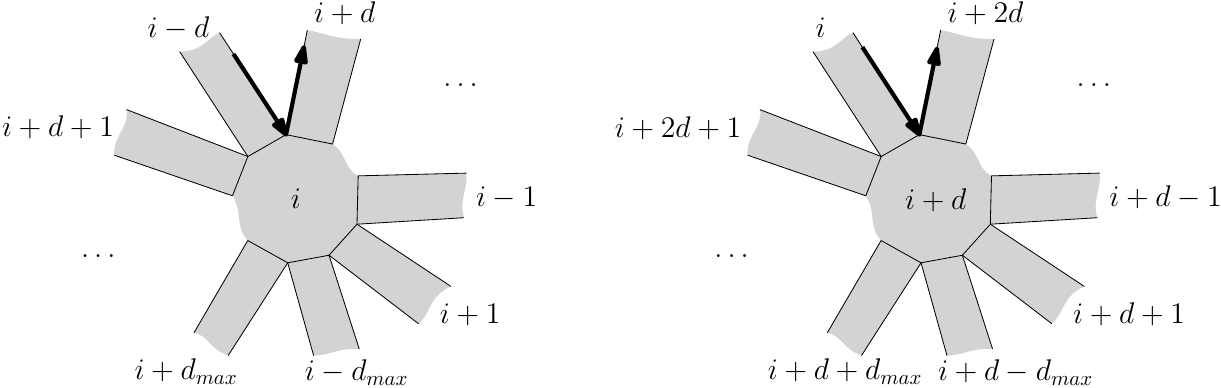}
	\caption{Even cycle.}
	\label{figure:even-cycle}
	\end{figure}

Case $2$: the boundary component arrives to $i$ along the left (with respect to $i$) side of the strip, Figure~\ref{figure:odd-cycle}.
In this case it then goes to $i+d+1$ and again arrives along the left side of the strip, and so on. So, it visits the vertices \[i-d, i+(d+1), i+(d+1)+(d+2), i+(d+1)+(d+2)+(d+3), \ldots \] We call such boundary component and the corresponding cycle in the graph \emph{odd}, as most of the gaps between consecutive vertices in the same component of $G$ are odd. Note, that because $1$ follows $d_{max}$ on the disks boundaries, at some point an odd boundary component might go through vertices $\ldots j, j+d_{max}, j+d_{max}+1, j+d_{max}+1+2 \ldots$ - this is the only case when the gap between consecutive vertices in the same component of $G$ in an odd cycle is even, $d_{max}+1$.

    \begin{figure}[ht]
	\center
	\includegraphics[width=1.0\linewidth]{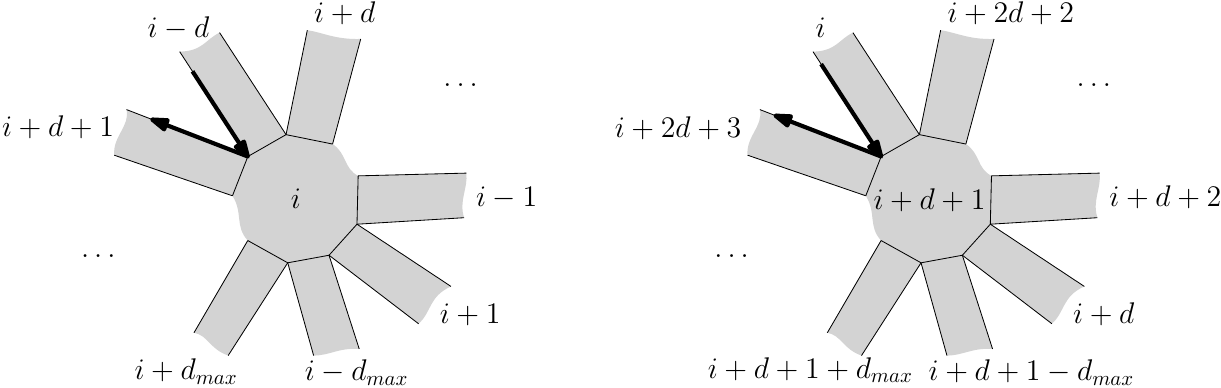}
	\caption{Odd cycle.}
	\label{figure:odd-cycle}
	\end{figure}

{\bf Checking properties (I), (II), (III).}

For property (I), we need to prove that for any vertices $j,k$ the sequence of three vertices of the form $j,?,k$ may occur at most once among all even cycles and at most once among all odd cycles.

If the cycle is even, then we immediately know that $k=j+2d$ for some $d$ and hence $?=j+d$. The only time when we encounter $j,j+d,j+2d$ in an even cycle is when we go first along the right side of the strip $j,j+d$ then along the left side of the strip $j+d,j+2d$, both times with respect to $j+d$, see Figure~\ref{figure:even-cycle} ($i$ in the figure is $j+d$). This uniquely determines a connected component of $\partial M_0$, and hence both the cycle and the part of it where we encounter the sequence.

If the cycle is odd, then there are two possibilities. If $k-j$ is odd, meaning that $k-j=2d+1$ for some $1\leq d < d_{max}$, then we know that $?=j+d$ and $k=j+d+(d+1)$; and if $k-j$ is even, then $?=j+d_{max}$ and $k=j+d_{max}+1$. As before, we can only encounter any of these two sequences once among all odd cycles, see Figure~\ref{figure:odd-cycle} ($i$ in the figure is $j+d$). 

For property (II), note that a vertex cannot appear twice in a cycle in less than $n/d_{max}>10$ steps.

For property (III), note that every edge of $G$ belongs to $2$ cycles, so the total length of the cycles is $O(n^2)$. So, to satisfy property (III) it is enough to split the cycles into simple paths of length $\Omega(n)$.

Every even cycle has length $2n$ and is already simple. Indeed, it has the form $i-d, i+d, i+2d, i+3d, \ldots$ for some $i,d$. In $n$ steps this sequence arrives to $i+nd=i$, which, however, lies in the different part of the bipartite graph $G$ compared to the starting $i$, because the number of steps $n$ is odd. So, to return to the starting $i$ the cycle needs $2n$ steps and since $n$ and $d$ are coprime, the cycle is simple.

It remains to deal with the odd cycles.
Consider an odd cycle. For each vertex write down the difference between it and the previous one. We get an (almost) monotone sequence of numbers of the form
\[
\ldots, d, d+1,\ldots,d_{max},1,2,3,\ldots,d_{max},1,2,3,\ldots.
\]
On one hand this sequence must repeat with a period equal to the length of the cycle, on the other hand it repeats every $d_{max}$ steps. Which means that our odd cycle can be split into paths of the form:
\[
i, i+1, i+2, \ldots, i+d_{max}.
\]
The length of each of these paths is $d_{max}\in\Omega(n)$ and they are all simple. The latter follows because for any $1\leq d\leq d+c\leq d_{max}$ and a vertex $j$, the vertices $j$ and $j+d+(d+1)+\ldots+(d+c)$ are different, since $d+(d+1)+\ldots+(d+c)=\frac{(c+1)(2d+c)}{2}$ is not divisible by $n$.

\bibliography{cubic-sphere}
\bibliographystyle{abbrv}

\end{document}